\documentclass[11pt,twoside]{article}
\usepackage{bbm}
\usepackage{amsmath,amssymb}

\usepackage{hyperref}
\hypersetup{
    colorlinks,
    citecolor=black,
    filecolor=black,
    linkcolor=black,
    urlcolor=black}

\makeatletter
\newcommand*{\mint}[1]{%
  \mint@l{#1}{}%
}
\newcommand*{\mint@l}[2]{%
  \@ifnextchar\limits{%
    \mint@l{#1}%
  }{%
    \@ifnextchar\nolimits{%
      \mint@l{#1}%
    }{%
      \@ifnextchar\displaylimits{%
        \mint@l{#1}%
      }{%
        \mint@s{#2}{#1}%
      }%
    }%
  }%
}
\newcommand*{\mint@s}[2]{%
  \@ifnextchar_{%
    \mint@sub{#1}{#2}%
  }{%
    \@ifnextchar^{%
      \mint@sup{#1}{#2}%
    }{%
      \mint@{#1}{#2}{}{}%
    }%
  }%
}
\def\mint@sub#1#2_#3{%
  \@ifnextchar^{%
    \mint@sub@sup{#1}{#2}{#3}%
  }{%
    \mint@{#1}{#2}{#3}{}%
  }%
}
\def\mint@sup#1#2^#3{%
  \@ifnextchar_{%
    \mint@sup@sub{#1}{#2}{#3}%
  }{%
    \mint@{#1}{#2}{}{#3}%
  }%
}
\def\mint@sub@sup#1#2#3^#4{%
  \mint@{#1}{#2}{#3}{#4}%
}
\def\mint@sup@sub#1#2#3_#4{%
  \mint@{#1}{#2}{#4}{#3}%
}
\newcommand*{\mint@}[4]{%
  \mathop{}%
  \mkern-\thinmuskip
  \mathchoice{%
    \mint@@{#1}{#2}{#3}{#4}%
        \displaystyle\textstyle\scriptstyle
  }{%
    \mint@@{#1}{#2}{#3}{#4}%
        \textstyle\scriptstyle\scriptstyle
  }{%
    \mint@@{#1}{#2}{#3}{#4}%
        \scriptstyle\scriptscriptstyle\scriptscriptstyle
  }{%
    \mint@@{#1}{#2}{#3}{#4}%
        \scriptscriptstyle\scriptscriptstyle\scriptscriptstyle
  }%
  \mkern-\thinmuskip
  \int#1%
  \ifx\\#3\\\else_{#3}\fi
  \ifx\\#4\\\else^{#4}\fi
}
\newcommand*{\mint@@}[7]{%
  \begingroup
    \sbox0{$#5\int\m@th$}%
    \sbox2{$#5\int_{}\m@th$}%
    \dimen2=\wd0 %
    \let\mint@limits=#1\relax
    \ifx\mint@limits\relax
      \sbox4{$#5\int_{\kern1sp}^{\kern1sp}\m@th$}%
      \ifdim\wd4>\wd2 %
        \let\mint@limits=\nolimits
      \else
        \let\mint@limits=\limits
      \fi
    \fi
    \ifx\mint@limits\displaylimits
      \ifx#5\displaystyle
        \let\mint@limits=\limits
      \fi
    \fi
    \ifx\mint@limits\limits
      \sbox0{$#7#3\m@th$}%
      \sbox2{$#7#4\m@th$}%
      \ifdim\wd0>\dimen2 %
        \dimen2=\wd0 %
      \fi
      \ifdim\wd2>\dimen2 %
        \dimen2=\wd2 %
      \fi
    \fi
    \rlap{%
      $#5%
        \vcenter{%
          \hbox to\dimen2{%
            \hss
            $#6{#2}\m@th$%
            \hss
          }%
        }%
      $%
    }%
  \endgroup
}
\usepackage{mathrsfs}
\usepackage{amssymb}
\usepackage{amsmath}
\usepackage{amsthm}
\usepackage{amsfonts}
\usepackage{color}
\usepackage{graphicx}
\usepackage[active]{srcltx}
\usepackage{tikz}
\usepackage{pgflibraryarrows}
\usepackage{pgflibrarysnakes}

\usepackage[cp1252]{inputenc}

\usepackage{mathrsfs}
\usepackage{graphicx}

\usepackage[active]{srcltx}

\allowdisplaybreaks

\usepackage{titletoc}
\titlecontents{section}[0pt]{\addvspace{2pt}\filright}
              {\contentspush{\thecontentslabel\ }}
              {}{\titlerule*[8pt]{.}\contentspage}

\def\rr{{\mathbb R}}

\def\fz{\infty}
\def\az{\alpha}

\def\lz{\lambda}

\def\bdz{\Delta}
\def\ez{\epsilon}

\def\gz{{\gamma}}

\def\bint{{\ifinner\rlap{\bf\kern.35em--}
\int\else\rlap{\bf\kern.45em--}\int\fi}\ignorespaces}

\def\bbint{{\ifinner\rlap{\bf\kern.35em--}
\hspace{0.078cm}\int\else\rlap{\bf\kern.45em--}\int\fi}\ignorespaces}

\newtheorem{thm}{Theorem}[section]
\newtheorem{lem}[thm]{Lemma}
\newtheorem{prop}[thm]{Proposition}
\newtheorem{rem}[thm]{Remark}
\numberwithin{equation}{section}

\title
{\Large\bf  Regularity of stable radial  solutions
to semilinear elliptic equations in MEMS problems
\footnotetext{\hspace{-0.35cm}
\endgraf{2000 {\it Mathematics Subject Classification: 35J61, 35B65, 35B35.}}
\endgraf  {\it Key words and phases: MEMS problems, stable radial solution, regularity}
\endgraf
F. Peng was supported by National Key R\&D Program of China 2025YFA1018400.}}

\author{Fa Peng and Salvador Villegas}
\begin{document}

\arraycolsep=1pt
\allowdisplaybreaks
 \maketitle

\begin{center}
\begin{minipage}{13.5cm}\small
 \noindent{\bf Abstract.}\quad
This paper investigates  the regularity of stable radial solutions to
semilinear elliptic equations arising in MEMS problems, modeled by the Dirichlet problem
$-\bdz u=f(u)$ in the unit ball $B_1$, where the nonlinearity
$f\in C^1([0,1))$ is nonnegative and satisfies $\int^1_0f(s)\,ds=+\fz$. We focus on the case where
$f$ blows up as $u\to 1^{-}$. Micro-electro-mechanical systems (MEMS) are widely used devices in engineering and technology. Our main result establishes for dimensions
$2\le n\le 6$, every stable radial solution is regular, meaning
$\|u\|_{L^\fz(B_1)}<1$.  This result gives a positive answer to an open problem
posed by Bruera and Cabr\'e concerning the regularity of stable solutions
for singular nonlinearities without requiring a
Crandall-Rabinowitz type condition, at least in the radial case.

\end{minipage}
\end{center}


\section{Introduction}

This paper focuses on the issues associated with MEMS problem:
\begin{align}\label{d-p}
\left\{
\begin{aligned}
-\bdz u & = f(u)&\quad{\rm in}&\quad B_1 \setminus\{ 0\}  \\
u& =0 \quad&{\rm on}&\quad \partial B_1,\\
\end{aligned}
\right.
\end{align}
where $B_1$ is the unit ball centered at the origin in $\rr^n$($n\ge 2$),  $0\le u\le 1$ is a stable radial  solution, and $f:[0,1)\to [0,+\fz)$  is
$C^1$-function which blow-up at
$u=1$.  Throughout the paper, for each $0\le t\le 1$ we define $F(t):=\int^t_0 f(s)\,ds$ and require that $F(1)=\int^{1}_0f(s)\,ds=+\fz$.

Recall that a radial solution $u$ of \eqref{d-p}  is called stable if
$$\int_{B_1}f'(u)\xi^2\,dx\le \int_{B_1}|\nabla \xi|^2\,dx$$
for every $\xi \in C^\fz(B_1)$ with compact support in $B_1\backslash\{0\}$, where
$f'(s)$ denotes the derivative of $f$ at the point $s$. Furthermore, if
$\|u\|_{L^\fz(B_1)}<1$, we refer to $u$ as regular in the sense of Bruera and Cabr\'e \cite{bc}.
Notably, the blow-up of $f(u)$ at $u=1$ implies that the stable radial solution $u$ does not need to be regular. For instance, when $n\ge 7$ and $f(t)=\frac{n-1}{1-t}$, the radial function $1-|x|$ is always stable yet satisfies
$\|u\|_{L^\fz(B_1)}=1$. In contrast, for dimensions $2\le n\le 6$, this function is not a stable solution. For further details, see Bruera and Cabr\'e \cite{bc} and  Meadows \cite{m04}.

This example naturally leads to the following question: in dimensions $2\le n\le 6$,  is every stable radial solution of \eqref{d-p} regular?

In the radial case, we provide a complete answer to this question.
\begin{thm}\label{thm}
Let $0\leq f\in C^1([0,1))$ satisfying $F(1)=+\fz$.
Assume that $0\le u\le 1$ is a  stable radial solution to \eqref{d-p}.
If $2\le n\le 6$, then $u$ is regular and satisfies
$$\frac 12 u_r(1)^2\le F(\|u\|_{L^\fz(B_1)})\le  Cu_r(1)^2.$$
Additionally, if $f$ is nondecreasing, then
 $$F(\|u\|_{L^\fz(B_1)})\le C,$$
where  $C>0$ is a universal constant.
\end{thm}
We now provide some remarks on the main results.
\begin{rem}
\rm
(i) The dimension $n \leq 6$ in Theorem \ref{thm} is optimal: for $n \geq 7$ and $f(t) = \frac{n-1}{1-t}$, the function $1 - |x|$ is always stable but not regular. Additional examples can be found in Bruera and Cabr\'e \cite{bc}.

(ii) When dimension
$3\le n\le 6$, $F(1)=+\infty$ is necessary for regular stable solutions to exist for certain nonlinearities $f$; counterexamples can be found in  Bruera and Cabr\'e \cite{bc}. Indeed, Bruera and Cabr\'e \cite{bc}
constructed a singular stable solution $u(x)=1-|x|^{2/(1+p)}$ for $n\ge 3$ and for some
 $p\in (0,1)$ satisfying
$$-\bdz u=\frac{2}{1+p}\left(\frac{2}{1+p}+n-2\right)(1-u)^{-p}$$
with $F(1)<+\fz$. For dimension $n=2$ (without assuming $F(1)=+\infty$), Luo, Ye and Zhou \cite{lyz} showed that the stable radial solution $u$ of $-\bdz u+ c(x)\cdot \nabla u=f(u)$ is regular if $f$ is nonnegative, nondecreasing, and convex, where $c(x)$ is smooth vector function. Observe that when $c(x)$ is a zero vector, this reduces to the case $-\Delta u = f(u)$.

(iii) In Theorem \ref{thm} above, with slight modifications, it can be shown that the stable radial  solution is regular under the conditions $0\le u\le 1$, $0\le f$  and $F(1)=+\fz$,
without requiring the boundary condition: $u=0$ on $\partial B_1$.
\end{rem}

Let us now review some relevant work in this direction. In the non-radial setting with a nonlinearity $f(u)$ that is unbounded as
$u\to +\fz$, a pioneering study was first conducted by Crandall and Rabinowitz \cite{cr75}
 with
$f(t)=e^t$ and $f(t)=(1+t)^p$, for $p>1$. Subsequently, the boundedness of stable solutions
$u$ to $-\bdz u=f(u)$ in general domains $\Omega$ has been extensively investigated over the past three decades; see \cite{cc06,c10,n20,v13}. Notably, Cabr\'e, Figalli, Ros-Oton and Serra \cite{cfrs} made a breakthrough:
Suppose $f$ is locally Lipschitz, nonnegative, and nondecreasing, and $u\in W^{1,2}_0(\Omega)$
is stable solution of $-\bdz u=f(u)$ in $C^3$ domain $\Omega$, they showed that
$$\|u\|_{C^{0,\az}(\overline \Omega)}\le C(n)\|u\|_{L^\fz(\Omega)}\quad
2\le n\le 9$$
for some $\az=\az(n)\in (0,1)$. They also established the local H\"older continuity of
$u$ for $n\le 9$. Note that the dimension $n\le 9$ is optimal for nonlinearities
$f(u)$ that blow-up as $u\to \fz$. In addition, Cabr\'e provided a quantitative proof of these results in \cite{c23,c24}.
We also mention that, as established in
the work of Erneta \cite{e23},
the optimal boundary regularity of $u$ in $C^3$ domains attainable already in the class of
$C^{1,1}$-domains.

Recently, inspired by a key stability inequality established \cite{cfrs},
Bruera and Cabr\'e \cite{bc} investigated the case of a nonlinearity $f(u)$
that blows up as $u\to 1^{-}$. Specifically, under the  following Crandall-Rabinowitz type condition:
\begin{align}\label{gz}
\gz:=\liminf_{t\to 1^{-}}\frac{f(t)f''(t)}{f'(t)^2}>1
\end{align}
for all  $f\in C^2([0,1))$ satisfying $f\ge 0$,
$f'\ge 0$ and $F(1)=+\fz$, they established the following local estimate for all stable solutions
(not necessarily radial) up to optimal dimension:
$$\|u\|_{L^{\fz}(B_{1/2})}\le  F^{-1}(C\|u\|^2_{L^1(B_1)})<1\quad
\forall \,  2\le n\le 6,$$
where constant $C$ depending only on $n$ and $\gz$. Additionally, a forthcoming study by Figalli and Franceschini \cite{ff} provides an explicit upper bound for the Hausdorff dimension of the singular set for stable solutions expressed in terms of $f$, $f'$, and $f''$. The result covers both globally defined nonlinearities and nonlinearities with finite blow-up.

On the other hand, for smooth domains $\Omega$, Bruera and Cabr\'e \cite{bc} also derived global estimates for stable solutions
$u\in C^2( \Omega)\cap C^0(\overline \Omega)$  of the Dirichlet boundary value problem:
\begin{align*}
\left\{
\begin{aligned}
-\bdz u & = f(u)&\quad{\rm in}&\quad \Omega  \\
u& =0 \quad&{\rm on}&\quad \partial \Omega,\\
\end{aligned}
\right.
\end{align*}
and showed the following:
\begin{itemize}
\item[$\bullet$] If $n\le 2$, then $\|u\|_{L^\fz(\Omega)}<1$ without condition \eqref{gz}.

\item[$\bullet$] If $3\le n\le 6$ and $f$ satisfies the condition \eqref{gz}, then $\|u\|_{L^\fz(\Omega)}<1$.

\end{itemize}
It is worth noting that additional partial results related to this direction can be found in
\cite{ces,cgg,lyz,m04}. In particular, for dimension $n=2$ (without requiring $F(1)=+\fz$), Luo, Ye and Zhou \cite{lyz} proved that the stable radial solution $u$ of $-\bdz u+c(x)\cdot \nabla u=f(u)$
 is regular provided that $f$ is nonnegative, nondecreasing, and convex.  Motivated by these findings, Bruera and Cabr\'e \cite{bc} proposed the following open problem:
\begin{center}
{\bf Open problem}. Does regularity of stable solutions hold up to the optimal
dimension for singular nonlinearities under no Crandall-Rabinowitz type condition?
\end{center}
In the radial setting,
we provide a positive answer to this open problem in
Theorem \ref{thm}. However, the non-radial case remains open.

Theorem \ref{thm} has an important application to the Gelfand-type problem; see for example
in \cite{bv97,b03}. Precisely, let $f\in C^1([0,1))$ satisfy $f(0)>0$ and be nondecreasing, convex, and $F(1)=+\infty$.

Given a constant $\lz>0$ consider the nonlinear elliptic problem:
\begin{align}\label{g-dp}
\left\{
\begin{aligned}
-\bdz u & = \lz f(u)&\quad{\rm in}&\quad B_1  \\
u& =0 \quad&{\rm on}&\quad \partial B_1,\\
\end{aligned}
\right.
\end{align}
A fundamental question is to determine a  constant $\lz^{\star}>0$ such that the problem
\eqref{g-dp} has a unique $L^1$-solution $u^{\star}\in L^1(B_1)$ in the sense that
$$-\int_{B_1}u^{\star}\bdz \xi \,dx=
\lz^{\star}\int_{B_1}f(u^{\star})\xi\,dx\quad\forall \,  \xi\in C^{0,1}_c(B_1)$$
Such $L^1$-solution solution $u^{\star}$ is called extremal solution.
The uniqueness and existence of extremal solution were established by
Castorina, Esposito and Sciunzi \cite{ces}. Moreover, they
showed that  for every $0<\lz<\lz^{\star}$, there exist a stable solution
$u_{\lz}\in C^2(\overline B_1)$ such that
$$\lim_{\lz\to \lz^{\star}}u_{\lz}
=u^{\star}.$$
Combining this result with Theorem \ref{thm}, we conclude that the extremal solution is always regular in dimensions $2\le n\le 6$.

\begin{thm}\label{thm-2}
Let $f\in C^1([0,1))$ be nondecreasing, convex, and
$F(1)=+\fz$. Suppose that $0\le u^{\star}\le 1$ is the extremal solution to
\eqref{g-dp}. Then
$$\|u^{\star}\|_{L^\fz(B_1)}<1\quad \forall \, 2\le n\le 6.$$

\end{thm}

Finally, we outline the main ideas for proving the theorems.
Since Theorem \ref{thm-2} follows directly from Theorem \ref{thm},
it suffices to explain the proof strategy for Theorem \ref{thm}.

A key starting point is the following observation for general radial solutions:
$$
F(\|u\|_{L^\fz(B_1)})= (n-1)\int^1_0\frac{u_r(t)^2}{t}\,dt+\frac{u_r(1)^2}2.
$$
Since $u\in C^1(\overline B_1\backslash\{0\})$,
the main task is reduced to bounding the term $\int^1_0\frac{u_r(t)^2}{t}\,dt$. By a key
estimate due to Villegas \cite[Theorem 1.7]{v12} (see also \cite{cc06}), we deduce
\begin{align*}
\vert u_r(t)\vert \le C_n \vert u_r(1)\vert t^{-n/2+\sqrt{n-1}+1},
\quad \forall \, 0<t\leq 1.
\end{align*}
Observe that
$$-n/2+\sqrt{n-1}+1>0\quad{\rm if \ and\ only\ if}\quad 2\le n\le 6.$$
This implies that $u_r(t)$ is H\"older continuous near  $t=0$ and hence
the behavior of
 $t^{-1}u_r(t)^2$ is like  $t^{-\az}$ for some
$\az\in (0,1)$ whenever $2\le n\le 6$. Therefore, the integral
$\int^1_0\frac{u_r(t)^2}{t}\,dt$ is finite, and we conclude

$$F(\|u\|_{L^\fz(B_1)})\leq C_nu_r(1)^2<+\fz.$$
Combining this with the condition $F(1)=\fz$, we obtain $\|u\|_{L^\fz(B_1)}<1$.
Moreover, if $f$ is nondecreasing, the universal bound on $F(\|u\|_{L^\fz(B_1)})$ follows
from the upper bound on $|u_r(1)|$; see Lemma \ref{un-bd} for details.

\section{Proof of the main result}
In this section, we begin with following key property for all radial solutions
$u$ to \eqref{d-p}
\begin{prop}\label{key-prop}
Let $n\geq 2$, and  $ f\in C^1([0,1))$ satisfying $f\ge 0$ and
$F(1)=+\fz$. Suppose that  $0\le u\le 1$ is a radial solution (not necessarily stable) of \eqref{d-p}. Then

\begin{align}\label{key-prop1}
F(\|u\|_{L^\fz(B_1)})= (n-1)\int^1_0\frac{u_r(t)^2}{t}\,dt+\frac{u_r(1)^2}2.
\end{align}
\end{prop}
\begin{proof}
Note that $f\ge 0$ and hence $u(t)$ is strictly decreasing with
respect $t\in (0,1]$ (see Proposition \ref{w12} below). Then $u(t)$ attains its maximum at point $t=0$ and
hence $u\in C^2(B_1\backslash\{0\})$.  First we observe that
$$f(u(t))=-\bdz u=-\left(u_{rr}(t)+\frac {(n-1)}r u_r(t)\right) \quad \forall \,  0<t<1.$$
For all $0<\ez<\|u\|_{L^\fz(B_1)}$, making the change of variable
$s=u(t)$ we obtain
\begin{align*}
F(\|u\|_{L^\fz(B_1)}-\ez)&=\int^{\|u\|_{L^\fz(B_1)}-\ez}_0f(s)\,ds
=\int^{\eta_{\ez}}_1f(u(t))u_r(t)\,dt\\
&=\int^{1}_{\eta_{\ez}}
\left(u_{rr}(t)+\frac {(n-1)}r u_r(t)\right)u_r(t)\,dt\\
&=\frac 12[u_r(1)^2-u_r(\eta_{\ez})^2]
+(n-1)\int^1_{\eta_{\ez}}\frac{u_r(t)^2}{t}\,dt,
\end{align*}
where $\eta_{\ez}\in (0,1)$ and $\lim_{\ez\to 0}\eta_{\ez}=0$.

Therefore, if $u$ is regular then $\lim_{t\to 0} u_r(t)=0$, and \eqref{key-prop1} follows by letting $\ez\to 0$.

Otherwise, if $u$ is not regular, taking into account that
$$F(\|u\|_{L^\fz(B_1)}-\ez)\leq \frac 12 u_r(1)^2
+(n-1)\int^1_{\eta_{\ez}}\frac{u_r(t)^2}{t}\,dt,$$
and letting $\ez\to 0$, we obtain $\int^1_0\frac{u_r(t)^2}{t}\,dt=+\infty$, which also implies \eqref{key-prop1} in this case.

\end{proof}

In order to show Theorem \ref{thm}, we recall the following important result concerning stable radial solutions, which was established by Villegas \cite[Theorem 1.7]{v12}.
\begin{lem}\label{vi-lem}
Let $n\geq 2$ and $f:[0,1)\to [0,+\fz)$ be  a $C^1$ function satisfying $f\ge 0$. Let $0\le u\le 1$ be a stable radial  solution to
\eqref{d-p} in $W^{1,2}(B_1)$. Then there exists a constant $K_n$ depending only on $n$ such that
\begin{align}\label{vi-le1}
|u_r(t)|\le K_n \|\nabla u\|_{L^2(B_1\backslash B_{1/2})}t^{-n/2+\sqrt{n-1}+1},
\quad \forall \,  0<t<1/2.
\end{align}
\end{lem}
Before applying Lemma \ref{vi-lem}, we need to verify that
$u$ belongs to $u\in W^{1,2}(B_1)$. For this purpose, we state the following proposition.
\begin{prop}\label{w12}
Let $n\ge 2$ and $f:[0,1)\to[0,+\infty)$ be a $C^{1}$ function satisfying $f\ge 0$. Suppose that $0\leq u\leq 1$ is a classical nontrivial $C^{2}$  radial solution (not necessarily stable) of  $-\bdz u = f(u) \quad{\rm in} \quad B_1\backslash\{0\} $. Then

\begin{itemize}
    \item[(i)] $u_{r}(r)<0$ in $(0,1)$.
    \item[(ii)] $\displaystyle\lim_{r\to 0^+}r^{n-1}u_{r}(r)=0$.
    \item[(iii)] $u\in W^{1,2}(B_{1})$.
\end{itemize}
\end{prop}
\begin{proof}
(i) Suppose by contradiction that there exists $r_{0}\in(0,1)$ such that $u_{r}(r_{0})>0$. Since $f\geq 0$ we have that
$$
(r^{n-1}u_{r}(r))'=-r^{n-1}f(u(r))\leq 0\ \mbox{ in }(0,1).
$$
Hence $r^{n-1}u_{r}(r)\geq r_{0}^{n-1}u_{r}(r_{0}):=\alpha_{0}>0$. This implies that
$$
u_{r}(r)\geq\alpha_{0}\,r^{1-n}\quad\mbox{for every }r\in(0,r_{0}).
$$
Therefore
$$
u(r_{0})-u(\varepsilon)=\int_{\varepsilon}^{r_{0}}u_{r}(r)\,dr\geq\alpha_{0}\int_{\varepsilon}^{r_{0}}r^{1-n}\,dr,\quad\text{for every }\varepsilon\in(0,r_{0}).
$$
Letting $\varepsilon\to 0$ and taking into account the boundedness of $u$ we obtain a contradiction. Hence $u_{r}\leq 0$ in $(0,1)$ and using that $u$ is not constant we conclude $u_{r}<0$ in $(0,1)$.

\medskip

(ii) Since $-r^{n-1}u_{r}(r)$ is a positive nondecreasing function in $(0,1)$, we have that there exists $\displaystyle\lim_{r\to 0}-r^{n-1}u_{r}(r):=\beta_{0}\geq 0$.

Suppose by contradiction that $\beta_{0}>0$. Hence $-r^{n-1}u_{r}(r)\geq\beta_{0}$
 in $(0,1)$, obtaining a similar contradiction as in (i).

\medskip

(iii) Since $u$ is bounded, it remains to prove that $|\nabla u|\in L^{2}(B_{1})$. Using again that $-u_{r}>0$ and $-r^{n-1}u_{r}$ is nondecreasing we deduce
\begin{align*}
\int_{B_{1}}|\nabla u|^{2}\,dx
&=\int_{0}^{1}\omega_{n}r^{n-1}u_{r}(r)^{2}\,dr
 =\omega_{n}\int_{0}^{1}(-r^{n-1}u_{r}(r))(-u_{r}(r))\,dr\\
&\leq\omega_{n}\int_{0}^{1}(-u_{r}(1))(-u_{r}(r))\,dr
 \leq\omega_{n}|u_{r}(1)|.
\end{align*}

\end{proof}

As a direct consequence of Lemma \ref{vi-lem} and Proposition \ref{w12}, we have:
\begin{prop}\label{key-prop-2}
Let $n\geq 2$ and $f:[0,1)\to [0,+\fz)$ be  a $C^1$ function satisfying $f\geq 0$. Let $0\le u\le 1$ be a stable radial solution to
\eqref{d-p}. Then there exists a constant $C_n$ depending only on $n$ such that

\begin{align}\label{vi-le2}
\vert u_r(t)\vert \le C_n \vert u_r(1)\vert t^{-n/2+\sqrt{n-1}+1},
\quad \forall \, 0<t\leq 1.
\end{align}
\end{prop}

\begin{proof}
Since $t^{2n-2}u_r(t)^2$ is nondecreasing with respect to $t\in (0,1]$, then
we have
$$\|\nabla u\|^2_{L^2(B_1\backslash B_{1/2})}
=\omega_n\int^1_{1/2}t^{n-1}u_r(t)^2\,dt
\le \omega_n 2^{n-1}\int^1_{1/2}t^{2n-2}u_r(t)^2\,dt\le \omega_n \frac{2^{n-1}}{2}u_r(1)^2.$$

Using Lemma \ref{vi-lem} and taking into account that $u\in W^{1,2}(B_1)$,  we obtain \eqref{vi-le2} for $0<t<1/2$.

If $1/2\leq t\leq 1$, using again that $\vert u_r(t)\vert t^{n-1}$ is nondecreasing, it follows easily that

$$\vert u_r(t)\vert\leq \vert u_r(1)\vert t^{1-n}\leq  \vert u_r(1)\vert 2^{n/2+\sqrt{n-1}} t^{-n/2+\sqrt{n-1}+1},$$

\noindent which completes the proof.
\end{proof}
In addition, if $f$ is  nondecreasing, one can derive the universal bound
of $u_r(1)$, which is useful for establishing an upper bound of $F(\|u\|_{L^\fz(B_1)})$.
\begin{lem}\label{un-bd}
Let $0\le f\in C^1([0,1])$ be nondecreasing. Assume
that $0\le u\le 1$ is a stable radial solution to
\eqref{d-p}. Then
$$|u_r(1)|\le 2.$$

\end{lem}
\begin{proof}
Consider the function $\Psi(t)=-nt^{\frac{n-1}n}u_r(t^{\frac 1n})$,
$t\in(0,1]$.  A direct calculation leads to
$$\Psi'(t)=-[(n-1)\frac{u_r(t^{\frac 1n})}{t}+u_{rr}(t^{\frac 1n})]
=f(u(t^{\frac 1n})),\quad \Psi''(t)=\frac 1nf'(u(t^{\frac 1n}))u_r(t^{\frac 1n})t^{\frac 1n-1}.$$
As $f$ is nonnegative and nondecreasing, the function $\Psi$ is nonnegative, nondecreasing, and concave. Therefore, $\frac{\Psi(t)}t$ is nonincreasing for $t\in (0,1]$,
which becomes
$$nt^{\frac{n-1}n}|u_r(t^{\frac 1n})|\ge nt|u_r(1)|.$$
Hence
$$|u_r(t^{\frac 1n})|\ge t^{\frac 1n}|u_r(1)|,\quad \forall \, t\in(0,1],$$
which is clearly equivalent to
$$|u_r(t)|\ge t|u_r(1)|,\quad \forall \, t\in(0,1].$$
Finally, by $0\le u\le 1$ and $u(1)=0$ we conclude by integrating
$$1\ge \int^1_0|u_r(t)|\,dt\ge \int^1_0t|u_r(1)|\,dt=\frac 12 |u_r(1)|.$$
\end{proof}

Combining Propositions \ref{key-prop} and
\ref{key-prop-2}, we are ready to prove Theorem \ref{thm}.
\begin{proof}
[Proof of Theorem \ref{thm}] Let $2\le n\le 6$ and $0\le u\le 1$  be
a stable radial solution to \eqref{d-p}. Applying  Proposition \ref{key-prop-2} we deduce

$$\int^1_0\frac{u_r(t)^2}t\,dt\leq C_n^2 u_r(1)^2 \int^1_0 t^{-n+2\sqrt{n-1}+1} dt.$$

Note that $-n+2\sqrt{n-1}+1>-1$ if and only if $2\leq n\leq 6$. Hence the last integral is finite and we obtain

$$\int^1_0\frac{u_r(t)^2}t\,dt\leq\frac{C_n^2}{-n+2\sqrt{n-1}+2} u_r(1)^2.$$

Using Proposition \ref{key-prop} we deduce

$$\frac 12 u_r(1)^2\le F(\|u\|_{L^\fz(B_1)})\le  M_n u_r(1)^2,$$

\noindent where $M_n$ is a dimensional constant.

Finally, by $u\in C^1(\overline B_1\backslash B_{1/2})$, one gets
$F(\|u\|_{L^\fz(B_1)})<+\fz$ and so that $u$ is regular.  Furthermore,
if $f$ is nondecreasing, Lemma \ref{un-bd} tells us $|u_r(1)|\le 2$, and thus
$$F(\|u\|_{L^\fz(B_1)})\le M_n u_r(1)^2\le 4M_n.$$
Since we are considering a finite
number of dimensions ($2\le n\le 6$), we can take $4M_n\leq C$ in order to
obtain a universal constant $C>0$.
Hence we finish this proof.
\end{proof}
We apply now Theorem \ref{thm} to prove Theorem
\ref{thm-2}.
\begin{proof}[Proof of Theorem \ref{thm-2}]
Let $0<\lz<\lz^{\star}$ and let $u_{\lz}\in C^2(\overline B_1)$ be a stable radial
solution to \eqref{g-dp}. It follows from Theorem \ref{thm} that
$$F(\|u_{\lz}\|_{L^\fz(B_1)})\le C(n)
\left(\frac{d u_{\lz}(1)}{dr}\right)^2.$$
Using Lemma 2.3 in \cite{v12} by Villegas, one has that
$$F(\|u_{\lz}\|_{L^\fz(B_1)})\le C(n)
\left(\frac{d u_{\lz}(1)}{dr}\right)^2
\le C(n)\left(\frac{d u_{\lz}(1/2)}{dr}\right)^2
\le C(n)\int^{3/4}_{1/2}\left(\frac{d u_{\lz}(t)}{dr}\right)^2\,dt,$$
where we used the fact that $t^{2n-2}\left(\frac{d u_{\lz}(t)}{dr}\right)^2$
is nondecreasing in the last inequality. Since
$\lim_{\lz\to \lz^{\star}}u_{\lz}=u^{\star}$ and
$\nabla u_{\lz}\to \nabla u^{\star}$ in $L^2(B_{3/4})$ as
$\lz \to \lz^{\star}$, passing to the limit $\lz\to \lz^{\star}$ we conclude that
$$F(\|u^{\star}\|_{L^\fz(B_1)})
\le  C(n)\int^{3/4}_{1/2}\left(\frac{d u^{\star}(t)}{dr}\right)^2\,dt<+\fz.$$

\end{proof}

\section*{Acknowledgments}
The authors are very grateful to Renzo Bruera and Professor Xavier Cabr\'e for
their previous suggestions regarding the need for revisions to the article's title,
as well as for all their valuable advice on this work.
Additionally, the authors extend their sincere thanks to Professor Yi Ru-Ya Zhang and
Yuan Zhou  for many helpful comments and valuable discussions related to the subject.
F. Peng was supported by National Key R\&D Program of China 2025YFA1018400.

\noindent Fa Peng

\noindent School of Mathematical Sciences, Beihang University, Changping District Shahe Higher Education Park South Third Street No. 9, Beijing 102206, P.R. China

\noindent{\it E-mail }:  \texttt{fapeng@buaa.edu.cn}

\bigskip

\noindent  Salvador Villegas

\noindent
Departamento de An\'{a}lisis Matem\'{a}tico, Universidad de Granada,
18071 Granada, Spain.

\noindent{\it E-mail }:  \texttt{svillega@ugr.es}

\end{document}